\documentclass[12pt]{article}
\usepackage{amssymb, amsfonts, amsmath, amsthm}

\newtheorem{theorem}{Theorem}
\newtheorem{proposition}[theorem]{Proposition}

\newtheorem{corollary}[theorem]{Corollary}
\newtheorem{conjecture}[theorem]{Conjecture}
\newcommand{\calc}{{\cal C}}
\newcommand{\cald}{{\cal D}}

\newcommand{\dd}{\displaystyle }
\newcommand{\ov}{\overline}

\def\0{\leqno}

\title{\bf On a generalization\\ of the Gauss's formula}
\author{Marius T\u arn\u auceanu}
\date{February 19, 2016}

\begin{document}

\maketitle

\begin{abstract}
In this paper we study a group theoretical generalization of the
well-known Gauss's formula that uses the generalized Euler's
totient function introduced in \cite{11}.
\end{abstract}

\noindent{\bf MSC (2010):} Primary 20D60, 11A25; Secondary 20D99,
11A99.

\noindent{\bf Key words:} Gauss's formula, Euler's totient
function, finite group, order of an element, exponent of a group.

\section{Introduction}

The {\it Euler's totient function} (or, simply, the {\it totient
function}) $\varphi$ is one of the most famous functions in number
theory. Notice that the totient $\varphi(n)$ of a positive integer
$n$ is defined to be the number of positive integers less than or
equal to $n$ that are coprime to $n$. The totient function is
important mainly because it gives the order of the group of all
units in the ring ($\mathbb{Z}_n$, +, $\cdot$). Alternatively,
$\varphi(n)$ can be seen as the number of generators or as the
number of elements of order $n$ of the finite cyclic group
($\mathbb{Z}_n$, +).

Recall also a well-known arithmetical identity involving the
totient function, namely the {\it Gauss's formula}
$$\dd\sum_{d\mid n}\varphi(d)=n, \hspace{1mm}\forall\hspace{1mm} n\in\mathbb{N}^*.\0(1)$$

Many generalizations of the totient function are known (for
example, see \cite{2,3,5,8} and the special chapter on this topic
in \cite{6}). From these, the most significant is probably the
{\it Jordan's totient function} (see \cite{1}).

The starting point for our discussion is given by the paper
\cite{11}, where a new group theoretical generalization of the
totient function has been studied. This is founded on the remark
that $\varphi(n)$ counts in fact the number of elements of order
$\exp(\mathbb{Z}_n)$ in ($\mathbb{Z}_n$, +). Consequently, it
makes sense to define
$$\varphi(G)=|\{a\in G \mid o(a)=\exp(G)\}|$$for any finite group $G$. It is
obvious that $\varphi(\mathbb{Z}_n)=\varphi(n)$, for all
$n\in\mathbb{N}^*$, and so a generalization of the classical
totient function $\varphi$ is obtained. We observe that for
$G\cong\mathbb{Z}_n$ the Gauss's formula can be rewritten as
$$\dd\sum_{H\leq\, G}\varphi(H)=|G|\,.\0(2)$$This leads
to the natural problem $$\textit{which are the finite groups $G$
satisfying the equality}\hspace{1mm} {\rm (2)\,?}$$Its study is
the main goal of the current paper. We show that the cyclic groups
are the unique abelian groups with this property. Inspired by some
particular cases, we conjecture that this is also true for
nilpotent groups. Moreover, we give examples of non-nilpotent
groups $G$ satisfying (2). Several open problems on this topic are
also formulated.

Most of our notation is standard and will not be repeated here.
Basic definitions and results on groups can be found in
\cite{4,9}. For subgroup lattice concepts we refer the reader to
\cite{7,10}.

\section{Main results}

For a finite group $G$ let us denote
\[ S(G)=\dd\sum_{H\leq\, G}\varphi(H)\,. \]
In this way, we are interested to describe the class $\calc$
consisting of all finite groups $G$ for which $S(G)=|G|$.

Obviously, the finite cyclic groups are contained in $\calc$, by
the Gauss's formula. On the other hand, we easily obtain
$S(\mathbb{Z}_{2}\times\mathbb{Z}_{2})=7\neq
4=|\mathbb{Z}_{2}\times\mathbb{Z}_{2}|$, proving that $\calc$ is
not closed under direct products or extensions.

For a detailed study of the class $\calc$, we must look first at
some basic properties of the map $S$. We remark that it satisfies
the inequality
$$S(G)\geq\dd\sum_{H\in C(G)}\varphi(H)=\dd\sum_{H\in C(G)}\varphi(|H|),\0(3)$$where
$C(G)$ denotes the poset of cyclic subgroups of $G$. Another easy
but very important property of $S$ is the following.

\begin{proposition}\label{th:C1}
    $S$ is multiplicative, that is
    if $(G_i)_{i=\ov{1,k}}$ is a family of finite groups of coprime
    orders, then we have:
    \[ S(\prod_{i=1}^k G_i)=\prod_{i=1}^k S(G_i). \]
\end{proposition}

\begin{proof}
Since the groups $(G_i)_{i=\ov{1,k}}$ are of coprime orders, we
infer that every subgroup $H$ of $G=\prod_{i=1}^k G_i$ can be
written as $H=\prod_{i=1}^k H_i$ with $H_i\leq G_i$,
$\forall\hspace{1mm}i=\ov{1,k}$. By Lemma 2.1 of \cite{11}, we
know that $\varphi$ is multiplicative and therefore
$$\varphi(H)=\prod_{i=1}^k \varphi(H_i)\,.$$Then one obtains
$$S(\prod_{i=1}^k G_i)=\dd\sum_{H\leq\, G}\varphi(H)=\dd\sum_{i=1}^k\dd\sum_{H_i\leq\, G_i}\varphi(H_1)\varphi(H_2)\cdots\varphi(H_k)=$$
$$\hspace{-10mm}=\prod_{i=1}^k\left(\,\dd\sum_{H_i\leq\,G_i}\varphi(H_i)\right)=\prod_{i=1}^k S(G_i)\,,$$as desired.
\end{proof}

In particular, Proposition 1 shows that the computation of $S(G)$
for a finite nilpotent group $G$ is reduced to $p$-groups.

\begin{corollary}
Let $G$ be a finite nilpotent group and $G_i$, $i=1,2,...,k$, be
the Sylow subgroups of $G$. Then
$$S(G)=\prod_{i=1}^k S(G_i).$$
\end{corollary}

\begin{proof}
The equality follows immediately from Proposition 1, since a
finite nilpotent group is the direct product of its Sylow
subgroups.
\end{proof}

Notice that for a finite abelian $p$-group $G$ the value
$\varphi(G)$ has been precisely computed in Theorem 2.3 of
\cite{11}. This is essential to show the following result.

\begin{theorem}\label{th:C1}
    Let $G$ be a finite abelian group. Then $S(G)\geq |G|$, and we have equality if and only if $G$ is cyclic.
\end{theorem}

\begin{proof}
Remark first that we can assume $G$ to be a $p$-group, by
Corollary 2. Let $(p^{\alpha_1}, p^{\alpha_2}, \dots,
p^{\alpha_r})$ be the type of $G$ and assume that $\alpha_1 \leq
\alpha_2 \leq \dots \leq \alpha_{s-1}<\alpha_s=\alpha_{s+1}= \dots
=\alpha_r$. Then we have
$$\varphi(G)=|G|\left(1-\frac{1}{p^{\hspace{0,5mm}r-s+1}}\right)\geq |G|\left(1-\frac{1}{p}\right)\,.$$On the other hand, it is
well-known that $G$ has $\frac{p^r-1}{p-1}$ maximal subgroups,
namely $p^{r-1}$ subgroups isomorphic to
$M_1=\mathbb{Z}_{p^{\alpha_1-1}}\times\mathbb{Z}_{p^{\alpha_2}}\times\cdots\times\mathbb{Z}_{p^{\alpha_r}}$,
$p^{r-2}$ subgroups isomorphic to
$M_2=\mathbb{Z}_{p^{\alpha_1}}\times\mathbb{Z}_{p^{\alpha_2-1}}\times\cdots\times\mathbb{Z}_{p^{\alpha_r}}$,
... , and one subgroup isomorphic to
$M_r=\mathbb{Z}_{p^{\alpha_1}}\times\mathbb{Z}_{p^{\alpha_2}}\times\cdots\times\mathbb{Z}_{p^{\alpha_r-1}}$.
One obtains $$\hspace{-16mm}S(G)\geq\varphi(G)+\dd\sum_{i=1}^r
p^{r-i}\varphi(M_i)+1\geq$$
$$\hspace{21mm}\geq |G|\left(1-\frac{1}{p}\right)+\dd\sum_{i=1}^r
p^{r-i}\,\dd\frac{|G|}{p}\left(1-\frac{1}{p}\right)+1=$$
$$\hspace{-14mm}=|G|\dd\frac{p^r+p^2-p-1}{p^2}+1\,.$$If $r\geq 2$, then
$$\dd\frac{p^r+p^2-p-1}{p^2}\geq\dd\frac{2p^2-p-1}{p^2}>1,$$implying
that $$S(G)>|G|+1>|G|\,.$$Consequently, $G$ belongs to $\calc$ if
and only if $r=1$, i.e. if and only if it is cyclic.
\end{proof}

Next we will focus on extending the above result from abelian
$p$-groups to arbitrary $p$-groups, and consequently to arbitrary
nilpotent groups. By a direct calculation, we infer that for all
non-abelian $p$-groups $G$ of order $p^3$ (whose classification is
well-known -- see e.g. \cite{9}, II) we have
$$S(G)>|G|\,.$$This inequality also holds for
other classes of non-abelian $p$-groups $G$, determined by the
existence of abelian subgroups of a given structure.

\begin{theorem}\label{th:C1}
    Let $G$ be a non-abelian $p$-group of order $p^n$, $n\geq 4$. If $G$ has an abelian subgroup of order $p^m$ and rank $r$ with $m+r\geq
    n+2$, then $S(G)>|G|$, i.e. $G$ is not contained in $\calc$.
    In particular, if $G$ has an elementary abelian maximal
    subgroup, then it does not belong to $\calc$.
\end{theorem}

\begin{proof}
Let $A$ be an abelian subgroup of order $p^m$ and rank $r$ of $G$,
and assume that $m+r\geq n+2$. By the proof of Theorem 3, we infer
that $$\hspace{-12mm}S(G)> S(A)\geq
p^m\dd\frac{p^r+p^2-p-1}{p^2}+1=$$
$$=p^{m+r-2}+p^{m-2}\left(p^2-p-1\right)+1\geq$$
$$\hspace{-4,5mm}\geq p^{m+r-2}+p^{m-2}+1>p^{m+r-2}\geq$$
$$\hspace{-56mm}\geq p^n,$$as claimed.
\end{proof}

\begin{theorem}\label{th:C1}
    Let $G$ be a non-abelian $p$-group of order $p^n$, $n\geq 4$. If $G$ has a cyclic maximal subgroup,
    then $S(G)>|G|$, i.e. $G$ is not contained in $\calc$.
\end{theorem}

\begin{proof}
By Theorem 4.1 of \cite{9}, II, we know that $G$ is isomorphic to
\begin{itemize}\item[--] $M(p^n)=\langle x,y\mid
x^{p^{n-1}}=y^p=1,\ y^{-1}x y=x^{p^{n-2}+1}\rangle$\end{itemize}
when $p$ is odd, or to one of the following groups
\begin{itemize}\item[--] $M(2^n)$
\item[--] the dihedral group $D_{2^n}$,
\item[--] the generalized quaternion group
$$Q_{2^n}=\langle x,y\mid x^{2^{n-1}}=y^4=1,\ yxy^{-1}=x^{2^{n-1}-1}\rangle,$$
\item[--] the quasi-dihedral group
$$S_{2^n}=\langle x,y\mid x^{2^{n-1}}=y^2=1,\ y^{-1}xy=x^{2^{n-2}-1}\rangle$$\end{itemize} when $p=2$.

A common property of all these $p$-groups $G$ is that they have
$p+1$ maximal subgroups, say $M_1$, $M_2$, ... , $M_{p+1}$, and
(at least) one of them is cyclic, say
$M_{p+1}\cong\mathbb{Z}_{p^{n-1}}$. Moreover, $\Phi(G)$ is cyclic
of order $p^{n-2}$. Then, by applying the Inclusion-Exclusion
Principle, one obtains
$$S(G)=\varphi(G)+\dd\sum_{i=1}^{p+1}S(M_i)-p\cdot S(\Phi(G))=\varphi(G)+\dd\sum_{i=1}^{p+1}S(M_i)-p^{n-1}.\0(4)$$

For $M(p^n)$ it is easy to check that $p$ maximal subgroups are
cyclic, say $M_i\cong \mathbb{Z}_{p^{n-1}}$, $i=2,3, ... , p+1$,
and $M_1\cong \mathbb{Z}_p\times\mathbb{Z}_{p^{n-2}}$. Then
$\varphi(M(p^n))=p\cdot\varphi(p^{n-1})=p^n-p^{n-1}$ and (4) leads
to
$$\hspace{-10mm}S(M(p^n))=p^n-p^{n-1}+S(\mathbb{Z}_p\times\mathbb{Z}_{p^{n-2}})+p\cdot p^{n-1}-p^{n-1}>$$
$$\hspace{17mm}>p^n-p^{n-1}+p^{n-1}+p^n-p^{n-1}=2\cdot
p^n-p^{n-1}>p^n,$$according to Theorem 3.

For every $G\in\{D_{2^n}, Q_{2^n}, S_{2^n}\}$ we have
$\varphi(G)=2^{n-2}$. Then (4) can be rewritten as
$$S(G)=2^{n-2}+S(M_1)+S(M_2).\0(5)$$The pair $(M_1,M_2)$ of maximal subgroups of $D_{2^n}$,
$Q_{2^n}$ and $S_{2^n}$ is $(D_{2^{n-1}},D_{2^{n-1}})$,
$(Q_{2^{n-1}},Q_{2^{n-1}})$ and $(D_{2^{n-1}},Q_{2^{n-1}})$,
respectively. Clearly, in the first two cases (5) becomes a
recurrence relation which easily leads to
$$S(D_{2^n})=2^{n+1}+(n-3)\cdot2^{n-2}>2^n$$and
$$S(Q_{2^n})=(n+4)\cdot2^{n-2}>2^n,$$while for $G=S_{2^n}$ one
obtains
$$S(S_{2^n})=2^{n-2}+S(D_{2^{n-1}})+S(Q_{2^{n-1}})=(2n+9)\cdot2^{n-3}>2^n.$$This
completes the proof.
\end{proof}

Inspired by the previous results, we came up with the following
conjecture.

\begin{conjecture}\label{th:C1}
    Let $G$ be a finite nilpotent group. Then $S(G)\geq |G|$, and we have equality if and only if $G$ is cyclic.
\end{conjecture}

Obviously, Conjecture 6 can be reformulated in the next way:
\textit{the cyclic groups are the unique finite nilpotent groups
contained in $\calc$}. It leads to the natural assumption that
$\calc$ consists in fact only of the finite cyclic groups. This is
not true, as shows the following elementary example.

\bigskip\noindent{\bf Example.} Let $G$ be the non-abelian group
of order $pq$, where $p<q$ are primes and $p\mid q-1$. The
subgroup structure of $G$ is well-known: it possesses one subgroup
of order 1, $q$ subgroups of order $p$, one subgroup of order $q$
and one subgroup of order $pq$. Then
$$S(G)=1+q\,\varphi(\mathbb{Z}_p)+\varphi(\mathbb{Z}_q)+\varphi(G)=1+q(p-1)+q-1=pq=|G|\,,$$i.e. $G$ belongs to $\calc$.
\bigskip

In particular, the above example shows that the dihedral group
$D_6$ is contained in $\calc$. In fact we are able to characterize
the containment to $\calc$ for arbitrary dihedral groups $D_{2n}=
\langle x,y\mid x^n=y^2=1,\ yxy=x^{-1} \rangle$, $n\geq 2$.

\begin{theorem}\label{th:C1}
    The dihedral group $D_{2n}$ belongs to $\calc$ if and only if $n$ is odd.
\end{theorem}

\begin{proof}
Let $n=2^k m$ with $k,m\in\mathbb{N}$ and $m$ odd. Then the
lattice of divisors of $n$ can be written as the union of the sets
$\cald_i=\{\,2^i m'\hspace{1mm}\mid\hspace{1mm} m'\!\mid\! m\}$,
$i=0,1,...,k$. On the other hand, for every divisor $d$ of $n$,
$D_{2n}$ has one subgroup isomorphic to $\mathbb{Z}_d$, namely
$\langle x^{\frac{n}{d}}\rangle$, and $\frac{n}{d}$ subgroups
isomorphic to $D_{2d}$, namely $\langle
x^{\frac{n}{d}},x^{i-1}y\rangle$, $i=1,2,...,\frac{n}{d}$\,.
Recall that we have $\varphi(D_2)=1$, $\varphi(D_4)=4$, and
$$\varphi(D_{2n})=\left\{\begin{array}{lll}
        0,&n \equiv 1 \hspace{1mm}({\rm mod}\hspace{1mm} 2)\\
        &&\hspace{1mm}\forall\hspace{1mm} n\geq 3\\
        \varphi(n),&n \equiv 0 \hspace{1mm}({\rm mod}\hspace{1mm} 2)\end{array}\right.$$by Theorem
2.6 of \cite{11}. It follows that
$$\hspace{-25mm}S(D_{2n})=\dd\sum_{H\leq D_{2n}}\varphi(H)=\dd\sum_{d\mid
n}\left(\varphi(\mathbb{Z}_d)+\dd\frac{n}{d}\,\varphi(D_{2d})\right)=$$
$$\hspace{20mm}=\dd\sum_{d\mid n}\varphi(\mathbb{Z}_d)+\dd\sum_{d\mid
n}\dd\frac{n}{d}\,\varphi(D_{2d})=\dd\sum_{d\mid
n}\varphi(d)+\dd\sum_{i=0}^k\sum_{d\in\cald_i}\dd\frac{n}{d}\,\varphi(D_{2d})=$$
$$\hspace{3mm}=n+\dd\sum_{m'\mid\, m}\dd\frac{n}{m'}\,\varphi(D_{2m'})+\dd\sum_{i=1}^k\sum_{m'\mid\, m}\dd\frac{n}{2^i
m'}\,\varphi(D_{2^{i+1}m'})=$$
$$\hspace{-72,5mm}=2n+\Sigma,$$where
$$\Sigma=\dd\sum_{i=1}^k\sum_{m'\mid\, m}\dd\frac{n}{2^i
m'}\,\varphi(D_{2^{i+1}m'}).$$Hence $S(D_{2n})=2n$ if and only if
$\Sigma=0$. This happens if and only if $k=0$, i.e. $n$ is odd.
\end{proof}

\noindent{\bf Remark.} By Theorem 7, we have $S(D_{2n})=2n$ for
$n$ odd. An explicit value of $S(D_{2n})$ for $n$ even can be
calculated, too. Let $n$ as above and let
$m=p_1^{\alpha_1}p_2^{\alpha_2}\cdots p_s^{\alpha_s}$ be the
decomposition of $m$ as a product of prime factors. We remark that
$\varphi(D_{2^{i+1}m'})=\varphi(2^i m')=2^{i-1}\varphi(m')$,
excepting the case $i=m'=1$ when $\varphi(D_{2^{i+1}m'})=3$. One
obtains
$$\hspace{10mm}\Sigma=\dd\frac{3n}{2}-\dd\frac{n}{2}+\dd\sum_{i=1}^k\sum_{m'\mid\, m}\dd\frac{n}{2m'}\,\varphi(m')=n+\dd\frac{kn}{2}\sum_{m'\mid\, m}\dd\frac{\varphi(m')}{m'}=$$
$$\hspace{-35mm}=n+\dd\frac{kn}{2}\prod_{i=1}^s\left(\alpha_i+1-\frac{\alpha_i}{p_i}\right)$$and
thus
$$S(D_{2n})=3n+\dd\frac{kn}{2}\prod_{i=1}^s\left(\alpha_i+1-\frac{\alpha_i}{p_i}\right).$$For
example, we can easily check that $$S(D_{12})=23.$$

Next we observe that both the non-abelian groups of order $pq$ and
the dihedral groups $D_{2n}$ with $n$ odd, which we verified to be
contained in $\calc$, are semidirect products of a cyclic normal
subgroup $N$ by a cyclic subgroup $H$ of prime order satisfying
$C_N(H)=1$. The containment of such a group to $\calc$ can be also
characterized, extending the above results.

\begin{theorem}\label{th:C1}
    Let $G$ be a finite non-abelian group and $N\cong\mathbb{Z}_n$ be a normal Hall subgroup of $G$ which has a complement $H$ of prime order $p$ such that $C_N(H)=1$.
    Then $G$ belongs to $\calc$ if and only if the number of complements of $N$ in $G$ is $n$.
\end{theorem}

\begin{proof}
Under our hypotheses, $L(G)$ consists of the subgroups of $N$, say
$N_d$ with $d=|N_d|$, $d\mid n$, of the complements of $N$ in $G$,
say $H_1=H$, $H_2$, ... , $H_{n_p}$, and of the semidirect
products $N_dH_i$, with $d\mid n$, $d\neq 1$ and $i=\ov{1,n_p}$.
Since $C_N(H)=1$, every $N_dH_i$ with $d\neq 1$ is not cyclic and
so it does not contain elements of order $dp=\exp(N_dH_i)$.
Consequently, we infer that $\varphi(N_dH_i)=0$ for all $d\mid n$
with $d\neq 1$ and all $i=\ov{1,n_p}$. This leads to
$$S(G)=S(N)+\dd\sum_{i=1}^{n_p}\varphi(H_i)=n+n_p(p-1).$$It is now obvious that
$$S(G)=np \Longleftrightarrow n_p=n,$$which ends the proof.
\end{proof}

We conclude that at least two important classes of finite groups
are contained in $\calc$: cyclic groups and semidirect products of
type indicated in Theorem 8. Remark that these groups $G$ are
supersolvable and that $S(G)$ equals the sum of all values of
$\varphi$ on the cyclic subgroups of $G$, that is (3) becomes an
equality.

Finally, we remark that every subgroup and every quotient of such
a group also belong to $\calc$, that is $\calc$ seems to be closed
under subgroups and homomorphic images.
\bigskip

We end this paper by indicating several natural problems on the
above class $\calc$.

\bigskip\noindent{\bf Problem 1.} Prove or disprove Conjecture 6.

\bigskip\noindent{\bf Problem 2.} Give a complete description of $\calc$ (in our opinion, it consists of the finite cyclic groups
and of non-abelian semidirect products of a certain type, most
probably metacyclic groups). It is true that $\calc$ is contained
in the class of finite supersolvable groups?

\bigskip\noindent{\bf Problem 3.} Study whether $\calc$ is closed
under subgroups and homomorphic images.

\vspace*{5ex}\small

\hfill
\begin{minipage}[t]{5cm}
Marius T\u arn\u auceanu \\
Faculty of  Mathematics \\
``Al.I. Cuza'' University \\
Ia\c si, Romania \\
e-mail: {\tt tarnauc@uaic.ro}
\end{minipage}


\begin{thebibliography}{00}
\bibitem{1} L. Dickson, {\it History of the theory of numbers}, I, Chelsea Publishing Co., New York, 1966.
\bibitem{2} P.G. Garcia and S. Ligh, {\it A generalization of Euler's $\varphi$-function}, Fibonacci Quart. {\bf 21} (1983), 26-28.
\bibitem{3} P. Hall, {\it The Eulerian functions of a group}, Quart. J. Math. {\bf 7} (1936), 134-151.
\bibitem{4} B. Huppert, {\it Endliche Gruppen}, I, Springer Verlag, Berlin-Heidelberg-New York, 1967.
\bibitem{5} P.J. McCarthy, {\it Introduction to arithmetical functions}, Springer Verlag, New York, 1986.
\bibitem{6} J. S\'andor and B. Crstici, {\it Handbook of number theory}, II, Kluwer A\-ca\-de\-mic Publishers, Dordrecht, 2004.
\bibitem{7} R. Schmidt, {\it Subgroup lattices of groups}, de Gruyter Expositions in Ma\-the\-ma\-tics 14, de Gruyter, Berlin, 1994.
\bibitem{8} R. Sivaramakrishnan, {\it The many facets of Euler's totient}, II, Nieuw Arch. Wisk. {\bf 8} (1990), 169-187.
\bibitem{9} M. Suzuki, {\it Group theory}, I, II, Springer Verlag, Berlin, 1982, 1986.
\bibitem{10} M. T\u arn\u auceanu, {\it Groups determined by posets of subgroups}, Ed. Matrix Rom, Bucure\c sti, 2006.
\bibitem{11} M. T\u arn\u auceanu, {\it A generalization of the Euler's totient function}, Asian-Eur. J. Math. {\bf 8} (2015), article ID 1550087.
\end{thebibliography}
\end{document}